\documentclass{amsart}
\usepackage{amssymb}
\usepackage{amsfonts}
\usepackage{amsmath}
\usepackage{amsthm}

\theoremstyle{plain}
\newtheorem{THM}{Theorem}
\newtheorem{COR}[THM]{Corollary}
\newtheorem{theorem}{Theorem}[section]

\newtheorem{cor}[theorem]{Corollary}
\newtheorem{prop}[theorem]{Proposition}
\newtheorem{lemma}[theorem]{Lemma}

\theoremstyle{definition}

\newcommand{\rank}{\mathrm{rank}}
\newcommand{\IW}{\mathrm{I}_{W} }

\begin{document}
\title[Integrable forms and varieties of minimal degree]{Germs of integrable
forms and varieties of minimal degree}
\author{Jorge Vit\'{o}rio Pereira}
\address{IMPA\\
Estrada Dona Castorina, 110\\
Rio de Janeiro / Brasil 22460-320}
\email{jvp@impa.br}
\author{Carlo Perrone}
\address{IMPA\\
Estrada Dona Castorina, 110\\
Rio de Janeiro / Brasil 22460-320}
\email{perrone@impa.br}
\date\today

\begin{abstract}
We study the subvariety of  integrable $1$-forms in a finite dimensional
vector space $W \subset \Omega^1(\mathbb C^n,0)$. We prove that the irreducible components
with dimension comparable with the rank of $W$ are of minimal degree.
%
%\bigskip
%
%\noindent{\sc Resum\'{e}.}
%On \'{e}tudie la sous-vari\'{e}t\'{e} form\'{e}e par les $1$-forms dans un espace vetoriel de dimension finie
%$W \subset \Omega^1(\mathbb C^n,0)$. On d\'{e}montre que les composantes irreductibles de cette sous-vari\'{e}t\'{e} dont
%la dimension est comparable au rang de $W$ sont de degr\'{e} minimal.
\end{abstract}

\maketitle

\section{Introduction}
Let $(\mathbb{C}^{n},0)$ be the germ of $\mathbb{C}^{n}$  at the origin.
For $q\in\{0,...,n\}$,  $\Omega ^{q}(\mathbb{C}^{n},0)$ will stand for the space of germs  of
holomorphic $q$-differential forms at $0  \in \mathbb{C}^{n}$.

In this work we are interested in describing the intersection of the set of integrable $1$-forms
in $\Omega^1(\mathbb C^n,0)$ with a finite dimensional vector space  $W \subset \Omega^1(\mathbb C^n,0)$.
In more concrete terms, our main objects of study are the projective varieties
\begin{equation*}
\mathrm{I}_{W} = \big\{ \, [\omega ]\in \mathbb{P}(W) \, \vert \, \omega \wedge d\omega =0 \, \big\}
\end{equation*}
where $W$ is as above and $\mathbb P(W)$ is the space of complex lines in $W$.

\medskip

Our motivation steams from the  study of the  irreducible components
of the space of foliations on $\mathbb P^n$, see    \cite{Cukierman-Pereira} and references therein. In the existing literature the usual approach to study the space of foliations on $\mathbb P^n$
passes through the recognition of distinguishing features of some classes of foliations,
and the proof of the stability of these features under small deformations. In this note, instead of looking at the foliations we focus directly on
the defining equations of $\mathrm{I}_{W}$.
For that sake we make  use of a simple idea  presented in \cite{Bouetou-Dufour}
reminiscent of Steiner's construction of rational normal curves, see Section \ref{S:steiner}.

\medskip

In order to state our main results we need first to introduce the {\bf rank} of a finite vector space $W \subset \Omega^1(\mathbb C^n,0)$. By definition,
$\rank(W)$  is the greatest integer $r$ for which the natural map
\[
   \bigwedge^r W \longrightarrow \Omega^r (\mathbb C^n,0)
\]
is not the zero map. Notice that $\rank(W) \le \min( \dim W, n)$.

\begin{THM}\label{T:1}
Let $W \subset \Omega^1(\mathbb C^n,0)$ be a finite dimensional vector space and let $\Sigma$ be an irreducible component of $\IW$.
If the codimension of $\Sigma$ in $\mathbb P( W)$ is at most $\rank(W)- 2$ then $\Sigma$ is a variety of minimal degree.
\end{THM}

Recall that a variety is  of minimal degree if its degree exceeds by one its codimension
in its linear span, that is
\[
X \text{ is of minimal degree} \iff \deg X = \dim \mathrm{Span}(X) - \dim X + 1 \, .
\]
They are well understood, and in particular  are completely
classified (see \cite{Eisenbud-Harris} and references therein).
Any variety of minimal degree is
either a linear subspace, a quadric hypersurface (eventually singular), a
rational normal scroll, the Veronese surface in $\mathbb{P}^{5}$ or a cone
over such a surface. When the rank and the dimension of $W$ coincide  we explore this classification to obtain the more precise result
below.

\begin{THM}\label{T:2}
If  $ \rank(W)= \dim W$ then every irreducible component of $\IW$ is either a linear subspace or a rational normal curve in its linear span.
\end{THM}

Theorem \ref{T:2} turns out to be sharp as the concrete examples in Section \ref{S:examples} testify. In Section \ref{S:GZ}
we characterize when a given rational normal curve of integrable $1$-forms is an irreducible component of $\IW$ using a beautiful
geometric construction due to Gelfand and Zakharevich, see Corollary \ref{C:3}.

\medskip

It has to be pointed out that the hypothesis on the rank is rather restrictive, and one should not expect similar
results about the space of foliations on projective varieties. For example, it is well known that for a fixed integer $d \ge 1$,
foliations induced by generic pencils of degree $d$ hypersurfaces in $\mathbb P^n$, $n \ge 3$, spread the irreducible component $R_n(d,d)$ of
the space of foliations of degree $2d -2$. Its codimension in  the projective space $\mathbb P H^0(\mathbb P^n, \Omega^1_{\mathbb P^n}(2d))$
is
\[
\underbrace{(n+1)N_{2d -1}   - N_{2d} + n-1 }_{\dim \mathbb P H^0(\mathbb P^n, \Omega^1_{\mathbb P^n}(2d))}
 \quad - \quad \underbrace{ 2N_d - 2  }_{\dim R_n(d,d)},   \text{ where  } N_{k} = \binom{n+k}{ k} -1
\]
while its degree, according to \cite[Section 5.1]{CPV}, is
\[
\frac{1}{N_d -1} \binom{2N_d -2}{N_d} \, .
\]
In particular, for $n$ or $d$ sufficiently large, it is clear that the degree is considerably greater than the codimension.
It does not seem to be easy to infer properties of the degree and/or geometry of the
irreducible components of the space of foliations on projective varieties from Theorem \ref{T:2}.
Nevertheless, at the other extreme of the spectrum of compact complex manifolds, there are the manifolds
of algebraic dimension zero. Recall that the algebraic dimension of compact complex manifold $X$, commonly denoted
by $a(X)$, is the transcendence degree over $\mathbb C$ of its field of meromorphic functions. For this class
of manifolds  Theorem \ref{T:2} has the following consequence.

\begin{COR}
Let  $X$ be a  compact complex manifold and $\mathcal L$ be a line-bundle
over it. If $a(X)=0$ then the irreducible components of the space of codimension one foliations with
conormal bundle $\mathcal L$ are either linear subspaces or rational normal curves.
\end{COR}
\begin{proof}
We are interested in the irreducible components of
\[
\left\{  [\omega] \in \mathbb P H^0 (X, \Omega^1_X\otimes \mathcal L) \, \big\vert \, \omega \wedge d \omega = 0 \right\} \, .
\]
Localizing at a generic point $x \in X$, the sections of $\Omega^1_X\otimes \mathcal L$ determine germs of
holomorphic $1$-forms that span a finite dimensional vector space $W$ of $\Omega^1(X,x) \simeq \Omega^1(\mathbb C^n,0)$  of dimension $m$.
If $\bigwedge^{m} W \to \Omega^{m} (X,x)$ is the zero map then there exists meromorphic functions $a_1, \ldots, a_{m} \in \mathbb C(X)$ and a basis $\omega_1, \ldots, \omega_{m}$ of $H^0(X, \Omega^1_X\otimes \mathcal L)$ such that $a_1 \omega_1 + \ldots + a_m \omega_m = 0$. But the  hypothesis $\mathbb C(X) = \mathbb C$ leads to
a contradiction
that implies $\dim W= \rank (W)$. The corollary follows from Theorem \ref{T:2}.
\end{proof}

\section{Rational normal curves and the proof of Theorem \ref{T:1}}\label{S:steiner}

\subsection{Steiner's construction of rational normal curves}
A rational normal curve  in $\mathbb P^n$ is nothing more than a smooth non-degenerate rational curve
of degree $n$. Up to projective automorphisms there is only one rational normal curve in $\mathbb P^n$, and
 it can be seen as the image of natural morphism
\begin{align*}
 \mathbb P^1 & \longrightarrow \mathrm{Sym}^n \, \mathbb P^1 \simeq \mathbb P^n \\
 p &\longmapsto \underbrace{p + \cdots + p}_{n \text{ times } } \, .
\end{align*}
Notice that this map is induced by the complete linear  system $| \mathcal O_{\mathbb P^1}(n) |$.

\medskip

Given a set of $n+3$ points in general position in $\mathbb P^n$,  that is no subset of $n+1$ points is contained in a hyperplane,
there is a unique rational normal curve containing it. This curve can be synthetically
constructed through the following procedure which can be  traced back to Steiner.
Let $p_1, \ldots, p_{n+3}$ be the $n+3$ points under consideration, and  for $i$ ranging from $1$
to $n$ let $\Pi_i$ be the $\mathbb P^{n-2}$ spanned by  the points $p_1, \ldots, p_{i-1},p_{i+1}, \ldots, p_n$.
For a fixed $i$ there is a pencil of hyperplanes containing $\Pi_i$. The elements of this pencil can
be written as $H_i(s:t) = \{ s F_i + t G_i = 0 \}$ where $(s:t) \in \mathbb P^1$ and $F_i, G_i$ are
linear forms on $\mathbb C^{n+1}$. These linear forms can be chosen  in order that $p_{n+1} \in H_i(0:1)$,
$p_{n+2} \in H_i(1:0)$ and $p_{n+3} \in H_i(1:1)$. It turns out that the map
\[
(s:t)  \longmapsto \bigcap_{i=1}^n H_i (s:t )
\]
parameterizes the unique rational normal curve through $p_1, \ldots, p_{n+3}$. Indeed {${(0:1)}$,} $(1:0)$ and $(1:1)$ are
mapped to $p_{n+1}, p_{n+2}$ and $p_{n+3}$ respectively. Furthermore, for $i = 1, \ldots, n$, there exists one and only one hyperplane in the pencil  $H_i(s:t)$ containing $p_i$, and $p_i$ belongs to
$H_j(s:t)$ for every $j \neq i$ and every $(s:t) \in \mathbb P^1$.

\subsection{Rational normal curves of integrable $1$-forms}

The following proposition  is a rephrasing of   the codimension one case
of \cite[Thm. 4.1]{Bouetou-Dufour}. The result, in codimension one as well as in arbitrary codimension,  is originally due to Panasyuk \cite{Panasyuk} and
settles a conjecture of Zakharevich \cite{Zaka}. In all these works
rational normal curves of integrable $1$-forms appear under the label of Veronese webs, a terminology
introduced in \cite{GZ}.

\begin{prop}\label{P:Chave}
Let  $W \subset \Omega^1(\mathbb C^{n+1}, 0)$ be a finite dimensional vector space with $\dim W =  \rank( W)$.
If there are $\dim W+2$ classes of integrable $1$-forms in general position in $\mathbb P(W)$  then the unique rational normal curve
through them parametrizes  classes of integrable $1$-forms.
\end{prop}
\begin{proof}
For  $\dim W \le 2$, the proposition is evident. So we will assume that $\dim W \ge 3$.
Moreover, after taking generic hyperplane sections,  we can also assume that $\dim W = \rank(W) =  n +1$.
Let
$p_1=[\omega_1], \ldots, p_{n+3}= [\omega_{n+3}]$ be $n+3$ points in general position  in $\mathbb P(W)$.
Since $\rank(W) = n+1$,  there exist germs of meromorphic vector
fields $v_1, \ldots, v_{n+1}$ satisfying
\[
\omega_i ( v_j ) = \delta_{ij} \,,  \quad   \quad  i, j \in \{ 1, \ldots, n+1 \} \, ,
\]
where $\delta_{ij}$ is the Kroenecker symbol. The hyperplanes in  $\mathbb P(W)$ are in one to one
correspondence with the lines in the space $V$ generated by $v_1, \ldots, v_{n+1}$. To wit, $V$  is
a concrete realization of the dual of $W$.

The hyperplanes $H_i(s:t)$
containing  $p_1, p_{i-1}, p_{i+1},  p_n$
are thus defined by the linear family of  vector fields
\[
\zeta_i(s,t) = s ( a v_i + b v_{n+1} ) + t (  c v_i + d v_{n+1} )
\]
where $a,b,c,d$ are complex numbers satisfying $ad-bc \neq 0$.  Hence
the unique rational normal curve through $[\omega_1],  \ldots, [\omega_{n+3}]$ is parametrized
by
\[
\langle \zeta_1(s,t) \wedge \cdots \wedge \zeta_{n}(s,t) \, , \, \omega_1 \wedge \ldots \wedge \omega_{n+1} \rangle \, . \,
\]
where $\langle \cdot, \cdot \rangle$ stands for the natural inner product.

Suppose now that the $1$-forms $\omega_1, \ldots, \omega_{n+3}$ are integrable. If this is the case
then for every $i, j \in \{ 1, \ldots, n \}$
\[
[ \zeta_i(s,t), \zeta_j(s,t) ] \wedge \zeta_1(s,t) \wedge \cdots \wedge \zeta_n(s,t)
\]
vanishes at $n+3$ distinct points $(s:t) \in \mathbb P^1$. But its coefficients have degree $n+2$ in the
variables $(s,t)$. Thus the above expression vanishes identically, which proves the proposition.
\end{proof}

\subsection{Proof of Theorem \ref{T:1}}Replace $W$ by a generic vector subspace $W'$ of dimension equal to
the codimension of $\Sigma$ plus two. Thus, since $W$ is generic,  $\dim W' = \rank(W')$ and $\mathbb P(W')$ intersects
$\Sigma$ at a curve $C$. Moreover, we can assume that $C$ is an irreducible component of $\mathrm{I}_{W'}$.

If $\Sigma$ is not of minimal degree then $C$ is also not of minimal degree. Replacing $W'$ by the linear span of
$C$ and applying Proposition \ref{P:Chave} to sufficiently many points in $C$ away from the other irreducible
components of $\mathrm{I}_{W'}$ one arrives at a contradiction which proves the theorem. \qed

\section{Varieties of minimal degree and the Proof of Theorem \ref{T:2}}

Suppose $W \subset \Omega^1(\mathbb C^n,0)$ is vector subspace satisfying $\dim W = \rank (W)$, and let
$\Sigma$ be an irreducible component of $\IW$ of dimension at least two.  Theorem \ref{T:1} implies that
$\Sigma$ is a variety of minimal degree. If it is not a linear subspace of $\mathbb P(W)$ then, after
replacing $W$ by a generic vector
subspace of appropriate dimension, we can assume that $\Sigma$ has dimension exactly two and it is a not
a plane linearly embedded in $\mathbb P(W)$. Moreover, it is harmless to  assume that
$\mathbb P(W)$ is the linear span of $\Sigma$.

To prove Theorem \ref{T:2} we aim at a contradiction. To obtain it we will analyze each of the classes
of surfaces of minimal degree. But first we recall in detail their  classification.

\subsection{Surfaces of minimal degree}
If $X \subset \mathbb P^n$ is a surface of  minimal degree
then $X$ is $\mathbb P^2$, or the embedding of $\mathbb P^2$ into $\mathbb P^5$ through the complete
linear system $| \mathcal O_{\mathbb P^2}(2)| \simeq \mathbb P^5$, or a rational normal scroll
$S(a,b)$ with $(a,b) \in \mathbb N^2 - \{ (0,0) \} $, and $a+ b +1 = n$.

\smallskip

The rational normal scrolls $S(a,b)\subset \mathbb P^{a+b+1}$ can be described as follows.
First consider two disjoint  linear subspaces $\mathbb P^a$ and $\mathbb P^b$ in   $\mathbb P^{n}$.
Consider now two rational normal curves $C_a \subset \mathbb P^a$ and $C_b\subset \mathbb P^b$, and let
$\varphi_a: \mathbb P^1 \to C_a$ and $\varphi_b: \mathbb P^1 \to C_b$ be their parametrizations. In case
$i=0$, $\varphi_i: \mathbb P^1 \to C_0 \subset \mathbb P^0$ is nothing more then the constant map. In all other
cases $\varphi_i$  is an isomorphic embedding. The rational normal scroll $S(a,b)$ is
the union of the lines $\overline{\varphi_a(t) \, \varphi_b(t)}$ for $t$ varying in $\mathbb P^1$. Note that
when $a=0$ we have a cone over a rational normal curve in $\mathbb P^{n-1}$.

\subsection{Veronese surface} We start the case by case analysis, excluding Veronese surfaces.

\begin{lemma}The surface $\Sigma$ is not a Veronese surface.
\end{lemma}
\begin{proof}
Assume $\Sigma$ is a Veronese surface.
Consider  eight points in general position contained in $\Sigma$ but not contained in any other
irreducible component of $\IW$. Let $C$ be the
unique rational normal curve  $C$ passing through them.

On the one hand $C$ is not contained in  $\Sigma$,
since $\deg C=5$ is odd and every curve in $\Sigma$ has even degree. Indeed, intersecting
a curve in $\Sigma$ with an hyperplane is the same as intersecting its pre-image under
the Veronese embedding $\mathbb P^2 \to \mathbb P^5$ with a conic.

On the  other hand, Proposition \ref{P:Chave} ensures that $C \subset \IW$. The choice of the eight points implies
$C$ must also be contained in $\Sigma$. This contradiction proves the lemma.
\end{proof}

\subsection{Pencils of integrable $1$-forms}
Now we turn our attention to the possibility of $\Sigma$ be a rational normal scroll.
We will first exclude the degenerate cases $\Sigma=S(0,n-1)$, $n\ge 3$. Notice that
these cases are characterized by their non-smoothness.

\begin{lemma} \label{L:lisa}
The surface $\Sigma$ is smooth.
\end{lemma}
\begin{proof}
If $\Sigma$ is not smooth then it must be the cone $S(0,n-1)$ over a rational normal curve in $\mathbb P^{n-1}$ with
$n\ge 3$.
The idea is to look at the line of integrable $1$-forms through the vertex of $S(0,n-1)$.
For that sake, let $\omega_0$ be a representative of the vertex and $\omega_1, \ldots, \omega_n$ be
representatives of points in $\Sigma$ away from the vertex  such that these $(n+1)$ differential forms
constitute a basis of $W$.

It will convenient to assume that all the non-zero $1$-forms in $W$ are non-zero at the origin.
Notice that this can be achieved after taking representatives and localizing outside the singular locus
of  $\omega_0 \wedge \ldots \wedge \omega_n$, which is non-zero thanks to the  assumption $\dim W = \rank (W)$.
Therefore there exists a choice of coordinates $x_0, \ldots, x_n$ in $\mathbb C^{n+1}$ for which
 $\omega_i = g_i d x_i$ where $g_0, \ldots, g_n$ are suitable germs of invertible functions.
Furthermore, after dividing all the $1$-forms by $g_0$, we can also assume that $\omega_0 = dx_0$.

For a fixed $i \in \{ 1, \ldots, n\}$, consider the linear family   $s \omega_0 + t \omega_i$
of integrable $1$-forms parametrized by $(s,t) \in \mathbb C^2$.
It is well known, see  for instance \cite{Cerveaupencil}, that there exists a unique meromorphic $1$-form $\eta_i$ such that
\[
d \left( s \omega_0 + t \omega_i \right)  =  \eta_i  \wedge \left( s \omega_0 + t \omega_i \right)
\]
for every $(s,t) \in \mathbb C^2$.
When $(s,t)=(1,0)$, the above  equation reads as $\eta_i \wedge dx_0 =0$.
The differentiation of this identity leads to $d\eta_i \wedge dx_0 = 0$. Combining these two identities with the
one obtained when $(s,t)=(0,1)$, one promptly
infers that
$
\eta_i = h_i(x_0,x_i) dx_0
$ for a suitable two variables function $h_i$.

Let now $\omega = \sum_{i=1}^n \lambda_i \omega_i$ be another integrable $1$-form distinct from the
previous ones. Of course, there exists such $1$-form since we are assuming that $\Sigma$ has dimension two.
As before we consider the linear family $s\omega_0 + t \omega$ and the corresponding $1$-form $\eta= h dx_0$ satisfying
\[d \omega = \eta \wedge \omega = \sum_{i=1}^n \lambda_i h f_i dx_0 \wedge dx_i \, . \]  Comparing this last identity with
\[
d \omega = \sum_{i=1}^n \lambda_i  d \omega_i  = \sum_{i=1}^n \lambda_i \eta_i \wedge \omega_i = \sum_{i=1}^n \lambda_i  h_i f_i dx_0 \wedge dx_i
\]
one deduces that $h_i = h = h(x_0)$.  Thus all the elements in $W$ are integrable contradicting the hypothesis that
$S(0,n-1)$, $n \ge 3$,  is an irreducible component of $\IW$.
\end{proof}

We have shown slightly more. The proof above  also shows  the following

\begin{lemma}\label{L:util}
If a rational normal scroll of the form $S(0,k)$, $k \ge 1$,  is contained in $\IW$ then
its linear span is also contained in $\IW$.
\end{lemma}

Notice that in the the extremal case $k=1$, $S(0,1)$ is nothing more than $\mathbb P^2$.

\subsection{Projections versus restrictions and the proof of Theorem \ref{T:2}}
To conclude the proof of Theorem \ref{T:2} it remains to consider the rational normal scrolls $S(a,b)$ with $a,b \ge 1$.
This is done in the next proposition.

\begin{prop}\label{P:util}
If a rational normal scroll of the form $S(a,b)$   is contained in $\IW$ then
its linear span is also contained in $\IW$.
\end{prop}
\begin{proof}
Assume $\mathbb P ( W )$  coincides with  the linear span of $S(a,b)$.
We will proceed by induction, with  the basis being given by Lemma \ref{L:util}.
To prove the result for $S(a,b)$, with $a,b \ge 1$,  assume  it holds for $S(a-1,b)$ and $S(a,b-1)$.

We can suppose, see the proof of Lemma \ref{L:lisa}, that  every non-zero  $1$-form in $W$ is non-zero
at the origin. Thus, if $\omega \in W$ is an integrable $1$-form then it defines a smooth
foliation $\mathcal F_{\omega}$ on  $(\mathbb C^{n+1},0)$. Let $L \simeq (\mathbb C^n,0)$ be  an arbitrary leaf of $\mathcal F_{\omega}$.
Notice that we are abusing the notation here. The leaf $L$ does not necessarily passes through the origin of $\mathbb C^{n+1}$. We are thinking in terms
of a representative of $\omega$   defined on a connected neighborhood of the origin where  the foliation $\mathcal F_{\omega}$ is defined by
a submersion with connected fibers, and  $L$ is an arbitrary fiber of such submersion.

If $\iota: L \to \mathbb C^{n+1}$ denotes its inclusion into $\mathbb C^{n+1}$,  then $W_L := \iota^* W$ is a vector space of $\Omega^1(\mathbb C^n,0)$
satisfying $\dim W_L = \dim W -1$ and $\rank(W_L) = \rank(W) -1$. The induced rational map
\[
\iota^*: \mathbb P( W ) \dashrightarrow \mathbb P ( W_L )
\]
is nothing more than the linear projection centered at $[\omega_0]$. Notice that $\mathrm{I}_{W_L}$ is contained
in the image of $\IW$.

Suppose  $S(a,b)$ is an irreducible component of  $\IW$ and that $[\omega]$ belongs either to $C_a$ or $C_b$ in $S(a,b)$. The projection
of $S(a,b)$ centered at a point in $C_a$, resp. $C_b$, is clearly $S(a-1,b)$, resp. $S(a,b-1)$. By induction hypothesis
$\mathrm{I}_{W_L}$ must coincide with $\mathbb P(W_L)$. Since $L$ is arbitrary, this implies that for every $\alpha \in W$ the
$4$-form $\omega \wedge \alpha \wedge d \alpha$ is identically zero.

Let $\omega_1, \ldots, \omega_4 \in W$ be four linearly independent $1$-forms with classes in $C_a \cup C_b$. The argument
above shows that for every $\alpha \in W$ and every $i \in \{ 1, 2, 3, 4\}$, the $4$-form $\alpha \wedge d \alpha \wedge \omega_i =0 $.
Thus $\alpha\wedge d \alpha = 0$ for any $\alpha \in W$. The proposition follows.
\end{proof}

\section{Examples}\label{S:examples}

\subsection{Left-invariant $1$-forms on Lie groups}
Let $G$ be a complex Lie group and $\mathfrak g$ be its Lie algebra.
The vector space of left-invariant $1$-forms on $G$ is naturally identified with $W=\mathfrak g^*$.
The classes of integrable $1$-forms in $\mathbb P W$ are in one to one correspondence with codimension one
Lie subalgebras of $\mathfrak g$.

For example, if  $\mathfrak{g}=\mathfrak{sl}(2,\mathbb{C})$ then the irreducible components of  $\mathrm{I}_{\mathfrak{g}
^{\ast}}\subseteq \mathbb{P}(\mathfrak{g}^{\ast})$ are easily described:
if $\alpha ,$ $\beta ,$ $\gamma $ is one basis  of $\mathfrak{g}^{\ast }$ satisfying $d\alpha =\alpha \wedge \beta ,$
$d\beta =\alpha \wedge \gamma ,$ $d\gamma =\beta \wedge \gamma $, then $\omega =x\alpha +y\beta + z\gamma \in \mathfrak{g}^{\ast}$ is
integrable if and only if
\begin{equation*}
(2xz-y^{2})\alpha \wedge \beta \wedge \gamma =0.
\end{equation*}
Thus $\mathrm{I}_{\mathfrak{g}^{\ast }}\subseteq \mathbb{P}(\mathfrak{g}^{\ast })$ has only one irreducible component which is a conic.

\smallskip

More generally, if $\mathfrak{g}$\ is any Lie algebra then main result of \cite{Hofmann}
implies that  the irreducible components of $\mathrm{I}_{\mathfrak{g}^{\ast }}\subseteq
\mathbb{P}(\mathfrak{g}^{\ast})$ are either linear subspaces or conics of
the type described above.

\subsection{Godbillon-Vey sequences}

Another natural source of rational curves of integrable $1$-forms
is the {\it development} of foliations with finite Godbillon-Vey sequence as studied in \cite{GV5}.  Given a meromorphic integrable $1$-form
on a projective manifold $X$ (or more generally pseudo-parallelizable compact manifold) there exists a sequence
of $1$-forms $(\omega_0, \omega_1, \ldots, \omega_k, \ldots)$ such that the {\it formal} $1$-form (defined on $X$ times a
formal neighborhood of the origin of $\mathbb C$)
\[
\Omega = dz + \sum_{i=0}^{\infty} \frac{z^i}{i!} \,  \omega_i \, ,
\]
is integrable and $\omega_0 = \omega$. A sequence with such properties is called a Godbillon-Vey sequence of
$\omega$, and $\Omega$ is a development of $\omega$. When this sequence is finite, {\it i.e.} $\omega_i=0$ for $i >  i_0$, the restriction of $\Omega$
to $\{ z = \text{const.} \}$ produces a rational normal curve of integrable $1$-forms in the vector space $W$ generated
by $\omega_0, \ldots, \omega_{i_0}$.

When $i_0 =2$ we are in a situation not essentially different from the example associate to $\mathrm{sl}(2, \mathbb C)$. In this
case the foliation induced by $\omega$ is transversely projective, and at neighborhood of a generic point of $X$ there is a
a map to $\mathrm{SL}(2,\mathbb C)$ such that the sequence $(\omega_0, \omega_1, \omega_2)$ is the pull-back of a sequence
of left-invariant $1$-forms on $\mathrm{SL}(2,\mathbb C)$.

When $i_0 > 2$, although one can obtain rational normal curves of degree equal to $\dim \mathbb P( W )$, no  example of this
kind fall under our hypothesis. Indeed, according to \cite[Lemma 2.3]{GV5},  $\omega_i\wedge \omega_j =0 $  for every $i,j \ge 2$.
In particular $\rank (W) \le 3$.

\subsection{Rational normal curves of arbitrary degree}
Fix an integer $n\ge 2$.  Set $\omega_0 = dx_0$ and, for $j$ ranging from $1$ to $n$, set
\[
 \omega_j = f_j dx_j \quad \text{ where } \quad f_{j}=(j+1) + j ( x_0 + \cdots + x_n ) \, .
\]
Consider the vector space $W \subset \Omega^{1}(\mathbb C^{n+1},0)$ generated by $\omega_0, \ldots, \omega_n$.
Clearly $\dim W = \rank (W) = n+1$.

Notice that $\omega_1, \ldots, \omega_n$ are all integrable $1$-forms. A computation shows that  $\omega_{n+1} = \sum \omega_i$   as well as $\omega _{n+2}=\sum_{i=0}^{n}%
\frac{1}{(i+2)}\omega _{i}$ are also integrable. Thus,  according to Proposition \ref{P:Chave},
the unique rational normal curve $C$ through $[\omega_0], \ldots, [\omega_{n+2}]$ is contained in
$\IW$. But, as another computation shows, the $1$-form $\sum_{i=0}^{n}(i+1)\omega _{i}$ is not integrable.
Hence Theorem \ref{T:2} implies that $C$ is an irreducible component of $\IW$.

\section{Gelfand-Zakharevich correspondence}\label{S:GZ}
Although concrete,  the previous example says nothing about the underlying geometry of rational normal curves
of integrable $1$-form. Here we are going to review a beautiful geometric construction from \cite[pages 79--80]{GZ2}, that puts
in correspondence  analytic equivalence classes of germs of holomorphic surfaces along  smooth rational curves endowed with a morphisms
to $\mathbb P^1$, and rational normal curves of integrable $1$-forms. Using this correspondence we will characterize when rational normal
curves are irreducible components of $\IW$ in terms of properties of the associated surface.

\subsection{From rational normal curves to surfaces}
Set $X$ equal to $(\mathbb C^{n+1},0)$. As above, $X$ should be thought as a sufficiently small
connected neighborhood of the origin. Let $W \subset \Omega^1(X)$
be a vector subspace of dimension and rank equal to $n+1$.
Suppose that all the non-zero $1$-forms in $W$ are non-singular at every point of $X$.

\medskip

Let  $\gamma : \mathbb P^1 \to \mathbb P(W)$ be a parametrization of a rational normal curve of
integrable $1$-forms in $\mathbb P(W)$. For any $\lambda \in \mathbb P^1$, let $\mathcal F_{\lambda}$
be the foliation associated to $\gamma(\lambda)$. Since the $1$-forms in $W$ have no singular
points, $\mathcal F_{\lambda}$ is a smooth foliation on $X$ and as such has leaf space naturally isomorphic
to $(\mathbb C,0)$. Considering the union of the leaf spaces of all the foliations $\mathcal F_{\lambda}$ with $\lambda$
varying in $\mathbb P^1$, one obtains a germ of complex surface $X^{(2)}$. To each point $x \in X$, there is
a smooth  rational  curve $C_x$ corresponding to the union of the leaves of the foliations $\mathcal F_{\lambda}$
through $x$. Let $C = C_0$ the curve corresponding to the origin $0 \in X$.
It is not hard to see that $C^2 = C \cdot C_x = n$: just take the point $x$ in the intersection
of  leaves of $\mathcal F_{\lambda_1}, \ldots, \mathcal F_{\lambda_n}$  through the origin.  Notice also that
$X^{(2)}$ comes endowed with a holomorphic map $\pi : X^{(2)} \to \mathbb P^1$ that associates to a leaf of
$\mathcal F_{\lambda}$ the point $\lambda \in \mathbb P^1$. Of course the restriction  $\pi_{|C}: C \to \mathbb P^1$
is an isomorphism.

\smallskip

If $\Gamma \subset X \times  X^{(2)}$ is the point-leaf correspondence, that is
\[
\Gamma = \left\{ ( x,L ) \in X \times X^{(2)} \, \big \vert \, x \in L \right\} \, ,
\]
and $\rho_1: \Gamma \to X$,  $\rho_2 : \Gamma \to  X^{(2)}$ are the natural projections then:
for any $\lambda \in \mathbb P^1$ and any leaf $L \subset X$ of $\mathcal F_{\lambda}$, $\rho_2 \rho_1^{-1}(L)$
is a point of $X^{(2)}$;  and
 for any section $\sigma: \mathbb P^1 \to X^{(2)}$, the intersection
\[
\bigcap_{\lambda \in \mathbb P^1} \rho_1\big( \rho_2^{-1} \big(\sigma(\lambda)\big)\big)
\]
is a point of $X$, see \cite[Theorem 2.2]{GZ2}.

The triple $(X^{(2)},C,\pi)$  will be called the Gelfand-Zakharevich triple associate to the rational normal curve
of integrable $1$-forms $\gamma(\mathbb P^1)$.  On the one hand the pair $(X^{(2)}, C)$, seen as a   germ of surface
along a rational curve modulo isomorphisms, does not depend on the parametrization of the rational normal curve.
On the other hand, the morphism $\pi$ does depend on the parametrization but its equivalence class modulo
composition on the left with automorphism of $\mathbb P^1$ does not. In other words, the linear system
that defines $\pi$ does not depend on the parametrization. Thus, it is fair to say that the Gelfand-Zakharevich
triple is canonically associated to the rational normal curve $\gamma(\mathbb P^1)$.

\subsection{From surfaces to rational normal curves} Start now with a triple
$(S, C, \pi)$, where $S$ is  a germ of smooth surface $S$ along a smooth rational
curve $C$ of self-intersection $n$ and endowed with a morphism $\pi : S \to \mathbb P^1$, and assume $\pi_{|C}  : C \to \mathbb P^1$ is a isomorphism.

Deformation theory tell us that the space of deformations $\mathcal X$ of $C$ is smooth and  $\mathcal X \simeq ( H^0(C, N_C), 0) \simeq (\mathbb C^{n+1}, 0)$.
To each $\lambda \in \mathbb P^1$, let $\mathcal F_{\lambda}$ be the foliation of $\mathcal X$ which has as leaves  deformations of $C$  intersecting
$\pi^{-1}(\lambda)$ at a fixed point. It is possible to show that there exists a vector space $W \subset \Omega^1(\mathcal X,0)$ of dimension and
rank equal to $n+1$, and a  family of foliations $\mathcal F_{\lambda}$  parametrized by
a rational normal curve of integrable $1$-forms contained in $\mathbb P(W)$, see \cite[Theorem 2.3]{GZ2}.
Hence, the two constructions just presented are inverse two each other modulo the respective natural equivalence relations.

\subsection{Rational normal curves as irreducible components of $\IW$}
Let now $W \subset \Omega^1(\mathbb C^{n+1},0)$ be a vector space of dimension
and rank equal to $n+1$, and  $C \subset \IW$ be a rational normal curve.
If $(S, C, \pi)$ is the Gelfand-Zakharevich triple associated to $C$ then
 $S$ has algebraic dimension\begin{footnote}{As
in the case of compact surfaces we are considering the algebraic dimension of $S$ as the  transcendence degree over
$\mathbb C$ of
its field of germs of meromorphic functions $\mathbb C(S)$.}\end{footnote} $a(S)$ equal to one or two.  Indeed, $a(S)$ it is at least one because the morphism $\pi: S \to \mathbb P^1$
induces a inclusion of $\mathbb C(\mathbb P^1)$ into $\mathbb C(S)$; and $a(S)$  is at most two
because $C^2 > 0$ what allow us to apply \cite[Th\'{e}or\`{e}me 6]{andreotti} or \cite[Theorem 6.7]{hartshorneCD}.

\begin{theorem}\label{T:3}
Assume $n\ge 2$.
The algebraic dimension of $S$ is two if and only if $\IW$ coincides with $\mathbb P(W)$.
\end{theorem}
\begin{proof}
If $\IW$ coincides with $\mathbb P(W)$ then the same arguments used to  prove Lemma \ref{L:lisa} imply
that $W$ is in a suitable system of coordinates the vector space generated by $h dx_0, \ldots, hdx_n$
for a fixed meromorphic function $h$.  In these coordinates, the foliations induced by elements of $W$
globalize to smooth foliations on $\mathbb C^{n+1}$. The leaf space of each one of these foliations
is isomorphic to $\mathbb C$, and the Gelfank-Zakharevich triple is isomorphic to $(E(\mathcal O_{\mathbb P^1}(n)), C_0, \pi)$
where $E(\mathcal O_{\mathbb P^1}(n))$ is the total space of $\mathcal O_{\mathbb P^1}(n)$, $C_0$ is the zero section, and
$\pi:E(\mathcal O_{\mathbb P^1}(n)) \to \mathbb P^1$ is the natural projection. Since $E(\mathcal O_{\mathbb P^1}(n))$ is an algebraic surface its  algebraic dimension
is at least two. Thus $\IW = \mathbb P(W)$ implies $a(S) = 2$.

Suppose now that $a(S)=2$. Therefore, there exists a projective surface $Z$ containing $S$ as an open subset. Moreover,
if $i: S \to Z$ is the inclusion then the Theorems of Andreotti and Hartshorne refereed to above imply that the induced
morphism $i^* : \mathbb C(Z) \to \mathbb C(S)$ is surjective. Thus the morphism $\pi: S \to \mathbb P^1$ extends
to a rational map, still denoted by $\pi$,   $\pi: Z  \dashrightarrow \mathbb P^1$. Since its indeterminacies, if any, are away from $C$,
it is harmless to assume that $\pi$ is indeed a regular map defined on all of $Z$.

We claim that the surface $Z$ is a rational surface and that the fibers of $\pi$ are rational curves.
The arguments are essentially the same as the ones laid down in \cite[Section 5.4.3]{PP} which we refer
for further details.  First notice that the abundance of rational curves on $Z$ implies that there are
no holomorphic $1$-forms on it. Hence linear and algebraic equivalence coincide thanks to Hodge theory.
After blowing-up  $Z$ at $n$ distinct points of $C$, one obtains a fibered surface $\overline  \pi : \overline Z \to \mathbb P^1$ containing
a section $\overline C$ of self-intersection zero which moves in a linear system of projective dimension one. This suffices
to show that fibers of $\overline \pi$, and hence also the fibers of $\pi$, are rational curves. Successive contractions of
the $(-1)$-curves on the fibers of $\pi$ that do not intersect the curve $C$ lead us to a relative minimal model $Z_0$ of
$Z$ which has to be the Hirzebruch surface $\mathbb P ( \mathcal O_{\mathbb P^1}(n) \oplus \mathcal O_{\mathbb P^1})$.
The complement of the section of self-intersection $-n$ is isomorphic to  $E(\mathcal O_{\mathbb P^1}(n))$ with the
  curve $C$ identified with $C_0$. Thus we conclude that the Gelfand-Zakharevich triple $(S,C,\pi)$ extends to the
  triple  $(E(\mathcal O_{\mathbb P^1}(n)), C_0, \pi)$ associate to $W = \oplus_{i=0}^n \mathbb C dx_i$. The naturalness
  of Gelfand-Zakharevich correspondence implies the result.
\end{proof}

\begin{cor}\label{C:3}Assume $n\ge 2$.
The curve $C$ is an irreducible component of $\IW$ if and only if the algebraic dimension of $S$ is one.
\end{cor}

\medskip
When $n=1$ all the elements of $\mathbb P(W)=\mathbb P^1$ correspond to  integrable $1$-forms.
Neverthless, there  is  a natural analogue of Theorem \ref{T:3} in this case. It   reads as:
$a(S)=2$ if and only if the there exists a closed meromorphic  $1$-form $\eta$  such that $d\omega = \eta \wedge \omega$
for every $\omega \in W$. The reader can easily infer such result from the proof of Theorem \ref{T:3}. Notice that
only the first paragraph has to be adapted,  the remaining  of the proof works as it is.

\begin{center}{\S}\end{center}

\bigskip

It would be interesting to investigate if, and if yes how, the Gelfand-Zakharevich correspondence globalizes when studying
rational normal curves of foliations on compact complex manifolds. For instance a structure theorem
for foliations in these curves along the lines of  \cite{Cerveaupencil} would be a welcome addition to the   literature.

\end{document}